\documentclass[a4paper,reqno,oneside,12pt]{amsart}

\usepackage{eucal}

\newcommand{\ddd}{\,\mathrm{d}}

\usepackage{bm}
\newcommand{\mathscr}[1]{\mathcal{#1}}

\def \rr {{\mathbb R}}
\def \cc {{\mathbb C}}

\def\var{\varphi}

\DeclareMathOperator{\dom}{dom}
\DeclareMathOperator{\spec}{spec}
\DeclareMathOperator{\supp}{supp}
\DeclareMathOperator{\pv}{p.v.}

\newcommand{\spece}{\spec_\mathrm{ess}}

\def\l{\lambda}

\def\ep{\varepsilon}

\def\k{\kappa}

\def\O{\Omega}

\def\S{\Sigma}

\def\sS{{\mathcal{S}}} 
\def\sK{{\mathcal{K}}}

\DeclareMathOperator{\Op}{Op}

\def \cc {{\mathbb{C}}}

\newtheorem{theorem}{Theorem}

\newtheorem{lemma}[theorem]{Lemma}

\theoremstyle{definition}

\newtheorem{remark}[theorem]{Remark}

\newtheorem{oprob}[theorem]{Open problem}

\title[]{Curvature contribution to the essential spectrum of Dirac operators with critical shell interactions}

\author[B. Benhellal]{Badreddine Benhellal}
\address{Carl von Ossietzky Universit\"at Oldenburg, Institut f\"ur Mathematik, 26111 Oldenburg, Germany}                
\email{badreddine.benhellal@uol.de}

\author[K. Pankrashkin]{Konstantin Pankrashkin}
\address{Carl von Ossietzky Universit\"at Oldenburg, Institut f\"ur Mathematik, 26111 Oldenburg, Germany}
\email{konstantin.pankrashkin@uol.de}

\begin{document}

\allowdisplaybreaks[3]

\begin{abstract} We discuss the spectral properties of three-dimensional Dirac operators with critical combinations of
electrostatic and Lorentz scalar shell interactions supported by a compact smooth surface.
It turns out that the criticality of the interaction may result in a new interval of essential spectrum.
The position and the length of the interval are explicitly controlled by the coupling constants and the principal curvatures of the surface. This effect is completely new compared to lower dimensional critical situations or special geometries
considered up to now, in which only a single new point in the essential spectrum was observed.
\end{abstract}



\maketitle	
%\tableofcontents
%================================================================================================
%================================================================================================

\section{Introduction}

The three-dimensional Dirac operator $D$ with fixed mass $m\in\rr$ acts on $\cc^4$-valued vector functions (spinors) $u$ as
\[
D:\quad u\mapsto -i\sum_{j\in\{1,2,3\}}\alpha_j\,\partial_j u +m\beta u,
\]
where $\partial_j$ is the partial derivative with respect to the $j$th variable (in any orthogonal Cartesian basis of $\rr^3$) and $(\alpha_1,\alpha_2,\alpha_3,\beta)$ is a family of four anticommuting $4\times 4$ hermitian matrices such that the square of each of them is the identity matrix. The maximal realization $H_0$ of $D$ in $L^2(\rr^3,\cc^4)$, given by
\[
H_0: u\mapsto D u,
\quad
\dom H_0:=H^1(\rr^3,\cc^4),
\]
is a self-adjoint operator in $L^2(\rr^3,\cc^4)$; it is the free Dirac operator playing a central role in relativistic quantum mechanics. In fact, many relativistic quantum Hamiltonians have the form $H_0+V$ with potential perturbations $V$: we refer to the book \cite{Tha} for a detailed discussion. 
The present paper is devoted to the spectral analysis of the
operators formally written as
\begin{equation}
	   \label{eq-hhh}
H:=H_0+(\ep +\mu\beta)\delta_{\S},
\end{equation}
where $\ep$ and $\mu$ are real parameters and $\delta_\S$
is the Dirac distribution supported by a compact surface $\S\subset \rr^3$. Such operators represent limiting models
of potential interactions strongly localized near $\S$ (the parameter $\ep$ is usually referred to as the strength of the electrostatic shell interaction supported by $\S$, while $\mu$ is the strength of the Lorentz shell interaction),
and they are rigorously defined using transmission conditions along $\S$. It seems that the operators of the above type were first studied in \cite{DES89} for the case when $\Sigma$ is a sphere, while
the case of general $\Sigma$ was introduced in \cite{AMV1,AMV2}.

In order to avoid technicalities we assume that $\Sigma$ is $C^\infty$ smooth (recall that it is also compact). By applying the Jordan-Brouwer separation theorem to each connected component of $\Sigma$ one finds a bounded open set $\Omega_+$ with smooth boundary $\partial\Omega_+=\Sigma$. We further set $\Omega_-:=\rr^3\setminus \overline{\Omega_+}$ and denote by
 $\nu$ the smooth unit normal on $\Sigma$ pointing to $\Omega_-$.
For a function $u$ defined on $\rr^3\setminus\S$ we denote by $u_\pm$
its restrictions on $\Omega_\pm$, and if $u_\pm$ admit suitably defined traces on $\S$, one defines $\delta_\S u$ to be the distribution defined by
\[
\langle \delta_\S u, \varphi\rangle :=\int_\S \dfrac{u_+ + u_-}{2}\,\varphi \ddd s
\]
for any test function $\varphi$ on $\rr^3$. For such a function $u$, the condition that the distribution $Du+(\ep +\mu\beta)\delta_{\S} u$
contains no summands with $\delta_\Sigma$ is then (at least formally) equivalent to
the transmission condition (compensation of singularities)
\begin{equation}\label{TC}
	i(\alpha\cdot\nu)\,(u_+ - u_-)
	+(\ep +\mu\beta) \frac{u_+ + u_-}{2}=0 \text{ on }\S,
\end{equation}
where we use the same symbol for functions and their traces on $\Sigma$ for a better readability, and
for $x=(x_1,x_2,x_3)\in\rr^3$ the traditional notation
\begin{equation}
	\label{ax}
\alpha\cdot x:=x_1\alpha_1 + x_2\alpha_2+x_3\alpha_3
\end{equation}
is used. 
Then by the operator $H$ formally written as \eqref{eq-hhh} we actually understand the linear operator $H$ in $L^2(\rr^3,\cc^4)$ acting as
\[
H:\ u\mapsto Du \text{ in the sense of distributions on }  \rr^3\setminus\S
\]
on the maximal domain
\begin{equation*}\label{domh}
\begin{aligned}
\dom H:={}&\big\{ u\in L^2(\rr^3,\cc^4):\ D u \in L^2(\rr^3\setminus \S,\cc^4)\\
& \text{ and $u$ satisfies the transmission condition \eqref{TC}}\big\};
\end{aligned}
\end{equation*}
recall that the traces of $u_\pm$ on $\Sigma$ are
then well-defined by duality as elements of $H^{-\frac{1}{2}}(\S,\cc^4)$.

The operator $H$ was already analyzed for a variety of situations, and main observations can be summarized as follows.
If $\ep^2-\mu^2\ne 4$, then the operator $H$ is self-adjoint. Moreover, its domain is contained in $H^1(\rr^3\setminus\S)$, the essential spectrum coincides with the spectrum of the free Dirac operator $H_0$, i.e. there holds
	$\spece H=\spec H_0\equiv\big(-\infty,-|m|\big]\cup \big[|m|,\infty\big)$,
	and the discrete spectrum of $H$ in $\big(-|m|,|m|\big)$ is at most finite.
In other words, the spectral picture of $H$ for $\ep^2-\mu^2\ne 4$ (which is usually referred to as the non-critical case) is a kind of classical one, i.e. it is a compactly supported perturbation which does not influence the essential spectrum and which produces at most finitely many discrete eigenvalues, see the papers \cite{BEHL2,BHSS, OBP} which also contains an overview of related results. As discussed in \cite{Rab1}, the non-critical case corresponds exactly to the situations in which the Lopatinski-Shapiro condition is satisfied.

It seems that the critical case $\ep^2-\mu^2=4$ was first approached in \cite{BH,OBV}: they only considered the purely electrostatic case $\ep=\pm 2$ and $\mu=0$ and have shown that $H$ is still self-adjoint but may have new portions of essential spectrum. Namely it is shown in \cite{BH} that for any $m\ne 0$ one has $0\in\spece H$ if $\Sigma$ is partially flat (i.e. if it contains a non-empty relatively open subset of a plane). The paper \cite{BB2} established the self-adjointness of $H$ for the general critical case (even under less restrictive regularity assumptions for $\Sigma$), but the question of the essential spectrum for general $\Sigma$ was not covered. The very recent preprint \cite{BHSS} shows the inclusion  $-\frac{ m \mu}{\ep}\in\spece H$ for partially flat $\Sigma$.

The operator $H$ admits a naturally defined counterpart in two dimensions, i.e. a two-dimensional Dirac operator with an interaction supported by a closed curve, which was studied in \cite{BHOP}. The respective family of operators is again parametrized by $(\ep,\mu)\in\rr^2$, and the same critical and non-critical cases arise. In the non-critical case the conclusions were
the same as in the three-dimensional case, but the critical case $\ep^2-\mu^2=4$ could be completely analyzed: it was
shown that the operator is self-adjoint on the natural maximal domain and that the criticality of the perturbation results
in the appearance of the new point $-\frac{m\mu}{\ep}$ in the essential spectrum (if $m\ne 0$). In other words, the compactly supported critical perturbation produces infinitely many eigenvalues in the initial gap of the essential spectrum.
We also mention that a complete characterization of the essential spectrum of  $H$ was established in~\cite{BB2} in three dimensions, but for a special unbounded $\Sigma$ (locally perturbed plane), and the same spectral effect was found, i.e., the appearance of the new point $-\frac{ m\mu}{\ep}$ for any critical combination of parameters.

In the present paper, we resolve completely the problem of the essential spectrum of $H$ for all critical combinations of parameters ($\ep^2-\mu^2=4$) and any smooth compact surface $\S\subset\rr^3$. The spectral picture turns out to be quite different from what could be predicted from all earlier observations, as shown in our main result: 

\begin{theorem}\label{main1}
	Let $\kappa_1$ and $\kappa_2$ be the principal curvatures on $\S$ and
	\[
	A_\S:=\max_{x\in \S} \big| \kappa_1(x)-\kappa_2(x)\big|.
	\]	
	If $\ep^2-\mu^2=4$, then $\spece H= \big(-\infty,-|m|\big]\cup \big[|m|,\infty\big) \cup J$ with
	\[	J:=\Big[-\frac{m\mu}{\ep}-\dfrac{A_\Sigma}{2|\ep|},-\frac{m\mu}{\ep}+\dfrac{A_\Sigma}{2|\ep|}\Big].
	\]
\end{theorem}
Therefore, for any critical perturbation of parameters the essential spectrum of $H$ consists of two (probably overlapping) portions having different nature. The first portion, $\big(-\infty,-|m|\big]\cup \big[|m|,\infty\big)\equiv \spec H_0$, is inherited from the free Dirac operator and does not depend on $\S$, while the second portion $J$ (and which only becomes visible for $m\ne 0$) is the interval of length $\frac{A_\Sigma}{|\ep|}$ centered at the point  $-\frac{m\mu}{\ep}$, and it strongly depends on the geometry of $\Sigma$. Remark that a similar structure of the essential spectrum appears in some mixed order problems \cite{grubb}.

It is a well known result of differential geometry that the only connected closed surfaces $\S\subset\rr^3$ with $\kappa_1\equiv \kappa_2$ (which is equivalent to $A_\S=0$ in our context) are spheres \cite[Sec.~3.2, Prop.~4]{docarmo}. Therefore, the case when $\Sigma$ is the union of finitely many disjoint spheres is the only one which produces a single new point in the essential spectrum, and
a critical shell interaction supported by any other $\Sigma$ adds a non-trivial interval to the essential spectrum.
It may be of its own interest to look for $\S$ that minimize  $A_\S$ under various constraints (which then minimize the length of the new portion of the essential spectrum), which leads to the following open problem:
\begin{oprob}
Which connected smooth surfaces $\S\subset\rr^3$ of genus $g\ge 1$ and fixed area minimize $\max |\kappa_1-\kappa_2|$, with $\kappa_1$ and $\kappa_2$ being the principal curvatures? 
\end{oprob}	

\begin{remark}
One can easily obtain some estimates for $A_\Sigma$ in terms of simpler geometric characteristics. The lower bound for the Willmore energy from \cite{mn} states that for any connected smooth compact surface $\S\subset \rr^3$ of genus $g\ge 1$ there holds
\begin{equation}
	\label{mn1}
\int_\S (\kappa_1+\kappa_2)^2\ddd  s\ge 8\pi^2.
\end{equation}
As the Euler characteristics of $\S$ is $\chi_\S:=2-2g$, the Gauss-Bonnet theorem gives
\begin{align*}
\int_\S (\kappa_1-\kappa_2)^2\ddd s
&=\int_\S (\kappa_1+\kappa_2)^2\ddd s-4\int_\S \kappa_1\kappa_2\ddd s\\
&=\int_\S (\kappa_1+\kappa_2)^2\ddd s-4\cdot 2\pi (2-2g)
\ge 8\pi^2+16 \pi (g-1).
\end{align*}
Using the trivial inequality
\begin{equation}
	 \label{mn1a}
A_\Sigma^2 |\Sigma|\ge \int_\S (\kappa_1-\kappa_2)^2\ddd s
\end{equation}
one arrives at the lower bound
\begin{equation}
	\label{mn2}
A_\Sigma\ge \dfrac{2\pi}{\sqrt{|\Sigma|}}\cdot \sqrt{2+\frac{4(g-1)}{\pi}}
\equiv
\dfrac{2\sqrt{2}\pi}{\sqrt{|\Sigma|}}\cdot \sqrt{1-\frac{\chi_\S}{\pi}}.
\end{equation}
However, this bound is never attained. In fact, if one had the equality in \eqref{mn2} for some $\Sigma$, then would have equalities in both \eqref{mn1} and \eqref{mn1a}. It is shown in \cite{mn} that the equality in \eqref{mn1} only holds for Clifford tori (some special tori of revolution). As the difference of principal curvatures of Clifford tori is non-constant, the inequality in \eqref{mn1a} is strict.
We further remark that the Clifford tori are not minimizers of $A_\S$ for $g=1$
and constant $|\S|$: a simple explicit computation shows that other tori of revolution produce strictly smaller values of $A_\S$.
\end{remark}

The rest of the text will be devoted to the proof of Theorem \ref{main1}, so from now on we assume that $\ep^2-\mu^4=4$. The inclusion $(-\infty,-|m|]\cup[|m|,\infty)\subset \spece H$  is easily shown by constructing Weyl sequences supported away from $\Sigma$ as in \cite[Thm. 5.7]{BH}, so we will only be interested in $\big(-|m|,|m|\big)\cap \spece H$ with $m\ne 0$. The proof structure is as follows. In Section~\ref{ss-prep} we recall
the reformulation of the spectral problem in terms of matrix singular integral operators over $\S$. Using the theory of block operator matrices we reduce our problem to the essential spectrum of a pencil of $2\times 2$ integral operators $S_\l$ over $\S$, see \eqref{lll3}. As the first step in the analysis of $S_\l$, in Section \ref{ss-imp} we compute the principal symbols of some auxiliary operators. These computations are then used in Section \ref{ss-schur} in order to complete the spectral study of $S_\l$, which completes the proof of Theorem~\ref{main1}.

\section{Preparations for the proof}\label{ss-prep}

Remark that there are many possible choices for $\alpha_j$ and $\beta$, but it follows from the theory
of Clifford algebras that the resulting operators are unitarily equivalent to each other and have the same spectra: see
e.g. \cite[Chapter 15]{dg} or the very condensed discussion in \cite[Lemma 2.4]{mobp}. To have more
explicit computations we choose  $(\alpha_1, \alpha_2, \alpha_3,\beta)$ as
\begin{align*}
	\alpha_k&=\begin{pmatrix}
	0 & \sigma_j\\
	\sigma_j & 0
	\end{pmatrix}\quad \text{for } j=1,2,3,
	\quad
	\beta=\begin{pmatrix}
	\mathit{I}_2 & 0\\
	0 & -\mathit{I}_2
	\end{pmatrix},
\end{align*}  
where $\mathit{I}_n$ is the $n\times n$-identity matrix and $(\sigma_1, \sigma_2, \sigma_3)$ are the $2\times 2$ Pauli matrices,
\begin{align*}
\sigma_1=\begin{pmatrix}
	0 & 1\\
	1 & 0
	\end{pmatrix},\quad \sigma_2=
	\begin{pmatrix}
	0 & -i\\
	i & 0
	\end{pmatrix} ,
	\quad
	\sigma_3=\begin{pmatrix}
	1 & 0\\
	0 & -1
	\end{pmatrix}.
	\end{align*}  
By analogy with \eqref{ax}, for $x=(x_1,x_2,x_3)\in\rr^3$ we denote
\[
	\sigma\cdot x:=x_1\sigma_1 + x_2\sigma_2+x_3\sigma_3,
\]
which satsfies the identity
\begin{equation}
	\label{ssxy}
 (\sigma\cdot x) (\sigma\cdot y)= \langle x,y\rangle I_2 + i\sigma\cdot (x\times y), \quad x,y\in\rr^3,
\end{equation}
where $\langle\cdot,\cdot\rangle$ and $\times$ are the usual scalar and cross vector product in $\rr^3$ respectively.

Recall that the Dirac operator $D$ satisfies $(D-\lambda)^2(D+\lambda)=-\Delta-(\lambda^2-m^2)$ and the function
\[
\Psi_z:x\mapsto \dfrac{e^{i\sqrt{z}|x|}}{4\pi |x|}, \quad \Im \sqrt{z}>0,
\]
is a fundamental solution of $-\Delta-z$. One easily checks that
for any $|\lambda|\le |m|$ the function $\Phi_{\l}:=(D+\lambda)(\Psi_{\lambda^2-m^2}\otimes I_4)$, i.e.
\begin{equation*}%\label{defsolo}
	\Phi_{\l}:\,x\mapsto \frac{e^{-\sqrt{m^2-\l^2}|x|}}{4\pi|x|}\left(\l I_4 +m\beta+\big( 1+ \sqrt{m^2-\l^2}|x|\big )\, i\alpha\cdot\frac{x}{|x|^2}\right),
\end{equation*}
is a fundamental solution of $D-\lambda$.  For the same $\l$ define the singular integral operators $C_\l$ on $L^2(\S,\cc^4)$ by
\begin{align*}%\label{CauchyOpe}
C_\l f(x)&=  \lim\limits_{\rho\searrow 0}\int_{\substack{y\in\S\\ |x-y|>\rho}}\Phi_{\l}(x-y)f(y)\ddd s(y),
\end{align*}
where $\ddd s$ stands for the integration with respect to the surface measure. It is
well known that  $C_\l:L^2(\S,\cc^4)\to L^2(\S,\cc^4)$ is bounded and self-adjoint \cite[Sec. 2]{AMV1},
satisfies
\begin{align}\label{PRI}
\big((\alpha\cdot \nu) C_\l\big)^2=\big( C_{\l} (\alpha\cdot \nu)\big)^2=-\frac{1}{4}\mathit{I}_4, \quad \l\in[-m,m],
\end{align}
see \cite[Lemma 2.2]{AMV2}, and standard considerations \cite[Chap. 7, Sec. 11]{T1} show that it is a zero-order classical pseudodifferential operator on $\S$, cf. \cite[Thoerem 4.1]{BBZ}.
The operator $C_\l$ is known to play a central role in the spectral analysis of $H$. To explain the link, let $\Delta_\Sigma$ be the negative Laplace-Beltrami on $\Sigma$ and
\[
\widetilde L:=(1-\Delta_\Sigma)^\frac{1}{4}
\]
viewed as a pseudodifferential operator, then
$\widetilde L:H^{s}(\S)\to H^{s-\frac{1}{2}}(\S)$ is an isomorphism for any $s\in\rr$. Consider now the operator
\[
\widetilde \Lambda_\lambda:=(\Tilde L\otimes I_4)\Big( \frac{1}{4}(\ep\mathit{I}_4 -\mu\beta) + C_\l \Big) (\Tilde L\otimes I_4)
\]
acting in $L^2(\S,\cc^4)$ on the maximal domain, Remark that $\widetilde\Lambda_\lambda$ is a first-order pseudodifferential opperator,
so it can be alternatively defined by starting on smooth functions and then taking the closure. Our analysis
will be based on the following equivalence proved in \cite[Lemma 4.1]{BB2}:
for any $\lambda\in\big(-|m|,|m|\big)$ one has the equivalence
\begin{equation}
	  \label{hl1}
\lambda\in \spece H \quad \Longleftrightarrow\quad 0\in\spece \widetilde\Lambda_\l.
\end{equation}
All subsequent analysis of $H$ will be based on this relation and on the study of $\widetilde\Lambda_\l$.

It will be convenient to consider the above $4\times 4$ operators as $2\times 2$ block operators with $2\times 2$ blocks. Namely,
\begin{gather}
\Phi_\l=\begin{pmatrix}
	(\l+m) k_\l I_2 & w_\l\\
	w_\l & (\l-m) k_\l I_2
\end{pmatrix},\nonumber\\
\label{kernels}
	k_{\l}:x\mapsto \frac{e^{i\sqrt{m^2-\l^2}|x|}}{4\pi|x|}, \
	\quad
	w_{\l}:x\mapsto \frac{e^{-\sqrt{m^2-\l^2}|x|}}{4\pi|x|}\left( 1+\sqrt{m^2-\l^2}|x|\right)i\sigma\cdot\frac{x}{|x|^2},
\end{gather}
which implies the block decomposition
\[
\label{Cauchymattform}
	C_{\l} = \begin{pmatrix}
		(\l+m)K_{\l}\otimes I_2  & W_{\l}\\
		W_{\l}&  (\l-m)K_{\l}\otimes I_2
	\end{pmatrix},
\]
with bounded self-adjoint operators $K_\l$ in $L^2(\S)$ and $W_\l$ in $L^2(\S,\cc^2)$ defined by
\begin{align}\label{Opkernels}
	\begin{split}
		K_{\l}g(x)&=\int_\S k_{\l}(x-y)g(y)\ddd s (y),\quad g\in L^2(\Sigma),\quad x\in\S,\\
		W_{\l}g(x)&= \lim\limits_{\rho\searrow 0}\int_{\substack{y\in\S\\ |x-y|>\rho}}w_{\l}(x-y)g(y)\ddd s(y),\quad g\in L^2(\Sigma,\cc^2), \quad x\in\S,
	\end{split}
\end{align}
and then the equality \eqref{PRI} reads as
\begin{align}\label{Pro W et K}
	\big((\sigma\cdot \nu)W_{\l}\big)^2 + (\l^2-m^2)\big((\sigma\cdot \nu)(K_{\l}\otimes I_2)\big)^2=-\frac{1}{4} I_2.
\end{align}

This gives the representation
\begin{equation}
	\label{lll}
\widetilde\Lambda_\l=\begin{pmatrix}
	\Big[\widetilde L\big( \tfrac{\ep-\mu}{4}+(\l+m)K_\l\big)\widetilde L\Big]\otimes I_2 & (\widetilde L\otimes I_2)W_\l (\widetilde L\otimes I_2)\\
	(\widetilde L\otimes I_2)W_\l (\widetilde L\otimes I_2) & \Big[\widetilde L\big( \tfrac{\ep+\mu}{4}+(\l-m)K_\l\big)\widetilde L\Big]\otimes I_2
\end{pmatrix}.
\end{equation}

Let us proceed with an interpretation of the blocks of $\Tilde\Lambda_\l$ as pseudodifferential operators. In order to fix various $2\pi$-like factors we remark that all conventions will correspond to the initial definition of the Fourier transform in $\rr^n$ by
\[
	\widehat{f}(\xi)=\int_{\rr^n}e^{-i\langle x,\xi\rangle} f(x)\ddd x
\]
for $f$ in the Schwartz class of functions, and we denote by
$\Op\sS^m_n$ the class of classical pseudodifferential operators of order $\le m$ acting on sections of $\S\times\cc^n$
and refer to \cite{T1,T2} for basic definitions and  properties of pseudodifferential operators on manifolds. It follows from the standard approach treating layer potential operators as pseudodifferential operators, see \cite[Chap. 7, Sec. 11]{T1}, that  $K_\l\in\Op\sS^{-1}_1$ and $W_\l\in\Op\sS^{0}_2$.

\begin{remark}[Special coordinates near a point of $\Sigma$]\label{rmk4}
We will compute all principal symbols for a specific choice of local coordinates on $\S$. Namely, let $x\in\S$ and consider the respective
Weingarten map $M_x:=-\ddd\nu|_x:T_x M\to T_x M$. The eigenvalues $\kappa_1(x)$ and $\kappa_2(x)$ of $M_x$ are called the principal curvatures of $\S$ in $x$. As $M_x$ is self-adjoint with respect
to the scalar product inherited from $\rr^3$, there exists an orthonormal basis $(e_1,e_2)$ in $T_x M$ (we omit its dependence on $x$)
such that $M_x e_j=\kappa_j(x)e_j$ for $j\in\{1,2\}$ and that the orthonormal basis $\big(e_1,e_2,\nu(x)\big)$ of $\rr^3$ is positively oriented. Then one can construct a local chart $\varphi:\rr^2\supset U\to V\subset\Sigma$ near $x$, with $0\in U$, $\varphi(0)=0$, $\partial_j\varphi(0)=e_j$ for $j\in\{1,2\}$. The coordinates defined by $\varphi$ will be referred to as \emph{special coordinates near $x$}.
Remark that the matrix of the metric tensor at $x$ in these coordinates coincides then with $I_2$, and the matrix
of $M_x$ is $\text{diag\,}\big(\k_1(x),\k_2(x)\big)$.
% Moreover,
%if $y=\varphi(t)\in V$, one obtains with the Help of the Taylor expansion:
%and $y=\var(t)\in U$, Taylor's expansion for $|y-x|\to 0$ (or, equivalently for $t\to 0$) implies
%\begin{align}
%	\label{dev1}
%y-x&= \var(t)-\var(0)= \sum_{j\in\{1,2\}}t_j e_j + \mathcal{O}(|t|^2),\\
%\nu(y)-\nu(x)&= \sum_{j\in\{1,2\}}t_j \partial_j\nu\big(\var (t)\big)_{| t=0} + \mathcal{O}(|t|^2)\nonumber\\
%&=- \sum_{j\in\{1,2\}}t_j M_x e_j + \mathcal{O}(|t|^2)= - \sum_{j\in\{1,2\}}\k_j(x)\,t_j e_j + \mathcal{O}(|t|^2).
%	\label{dev2}
%\end{align}
\end{remark}

The following properties of $K_\l$ are well known, but we include a proof for the sake of competeness (and as a warm-up for the subsequent constructions).
\begin{lemma}\label{lem5}
For any $|\l|\le |m|$ the operator $K_\l$ is positive and injective in $L^2(\Sigma)$, and
$K_\l\in\Op \sS^{-1}_1$ with principal symbol
\[
p_{K_\lambda}:(x,\xi)\mapsto \dfrac{1}{2|\xi|}
\]
in the above special coordinates near $x$.
Moreover,  $K_\l:H^s(\S)\to H^{s-1}(\S)$ is an isomorphism for any $s\in \rr$.
\end{lemma}	

\begin{proof} Consider the single layer potential $\sK_\lambda$ for $-\Delta+m^2-\lambda^2$,
\[
\sK_\l g (x)=\int_\S k_\l(x-y)g(y)\ddd s(y), \quad g\in L^2(\Sigma),\quad x\in\rr^3.
\]
and briefly recall its properties \cite[Chapter 9]{Mc}. One has $(-\Delta+m^2-\lambda^2)\sK_\l g=0$
in $\Omega_\pm$. For $|\l|=|m|$ one has $\nabla\sK_\l g\in L^2(\O_\pm)$ with $\lim_{|x|\to\infty} \sK_\l g(x)=0$,
while $\sK_\l g\in H^1(\O_\pm)$ for $|\l|<|m|$. The trace of $\sK_\l g$ on $\S$ coincides with $K_\l g$,
and one has the jump relation $\partial_\nu^+ K_\l g-\partial_\nu^- K_\l g=g$ on $\S$,
where $\partial_\nu^\pm$ is the outer normal derivative for $\Omega_\pm$.

Let $|\lambda|<|m|$, then using the Green formula we have
\begin{align*}
	(\lambda^2-m^2) \big\|\sK_{\l}g\big\|^2_{L^2(\Omega_\pm)}&=\big\langle \sK_{\l}g,-\Delta \sK_{\l}g\big\rangle_{L^2(\Omega_\pm)}\\
&=	\big\|\nabla\sK_{\l}g \|^2_{L^2(\Omega_\pm)}
	\mp \langle K_{\l}g, \partial^{\pm}_{\nu}\sK_{\l}g\rangle_{L^2(\S)}.
\end{align*}
Using the above jump relation one obtains then
\begin{align*}
	\langle K_{\l}g, g\rangle_{L^2(\S)}&=\langle K_{\l}g, \partial^{+}_{\nu}\sK_{\l}g\rangle_{L^2(\S)}-\langle K_{\l}g, \partial^{-}_{\nu}\sK_{\l}g\rangle_{L^2(\S)}\\
	&=\big\|\nabla\sK_{\l}g \|^2_{L^2(\Omega_+)}+\big\|\nabla\sK_{\l}g \|^2_{L^2(\Omega_-)}
	+(m^2-\lambda^2) \big\|\sK_{\l}g\big\|^2_{L^2(\rr^3)}\ge 0.
\end{align*}	
If $K_\lambda g=0$ for some $g\in L^2(\Sigma)$, then the last equation gives $\sK_\lambda g\equiv 0$,
and the jump relation on $\Sigma$ gives $g=0$.

For $|\l|=|m|$ we obtain in the same way
\begin{align*}
	\langle K_{\l}g, g\rangle_{L^2(\S)}=\big\|\nabla\sK_{\l}g \|^2_{L^2(\Omega_+)}+\big\|\nabla\sK_{\l}g \|^2_{L^2(\Omega_-)}\ge 0.
\end{align*}
If $K_\lambda g=0$ for some $g\in L^2(\Sigma)$, then $\nabla\sK_\l g=0$ in $L^2(\O_\pm)$  and it follows that $\sK_\l g$ is constant on each connected component of $\rr^3\setminus\S$.
Due to the decay at infinity, $\sK_\l g=0$ on the unbounded component, then the continuity along $\Sigma$ shows that $\sK_\l g$ is identically zero,
and the jump relation on $\Sigma$ gives $g=0$.

Let us compute the principal symbol of $K_\l$. For any convolution kernel $k(x,y)$ we will denote by $T_{k}$ the associted integral operator. First observe that
$k_{\l}(x)= k_{\pm m}(x) + k_{\l,1}(x)$, where $k_{\l,1}$ is bounded near $0$, which shows that
$T_{k_{\l,1}}\in\Op\sS^{-2}_1$. Hence, it is sufficient to compute the principal symbol
of $K_m$.

Let $x\in\S$ and choose a special coordinates $(U,V,\var)$ near $x$ as in Remark~\ref{rmk4}.
Let $\psi_j\in C^\infty_c(\Sigma)$, $j\in\{1,2\}$, be real-valued with $\supp \psi_j\subset V$.
Let $D\varphi$ be the Jacobi matrix of $\varphi$ and
\begin{equation}
	   \label{gg1}
G_\varphi(t):=D\varphi(t)^T D\varphi(t), \quad g_\varphi (t):=\sqrt{\det G_\varphi(t)}.
\end{equation}
Recall that for any $f\in C^\infty(\S)$ we have
\[
\psi_2\big(\varphi(s)\big) K_\lambda (\psi_1 f)\big( \varphi(s)\big)=
\psi_2\big(\varphi(s)\big)
\int_U k_\lambda\big(\varphi(s)-\varphi(t)\big) \psi_1\big(\varphi(t)\big) f\big(\varphi(t)\big) g_\varphi(t)\ddd t.
\]
For small $|s-t|$ we have
\begin{equation}
	\label{gg2}
\begin{aligned}
\big|\varphi(s)-\varphi(t)\big|&=\sqrt{\big\langle (s-t), G_\varphi(s)(s-t)\big\rangle}\Big( 1+O\big(|s-t|\big)\Big),\\
g_\varphi(t)&=g_\varphi(s)+O\big(|s-t|\big),
\end{aligned}
\end{equation}
so we can represent
\begin{equation}
	  \label{tmp1}
\begin{aligned}
	\big(\psi_2 K_{m}(\psi_1f)\big)\big(\varphi(s)\big) &=\psi_2\big(\varphi(s)\big)\,g_\varphi(s)\int_{U}\frac{\psi_1\big(\varphi(t)\big)f(\varphi(t))}{4\pi\sqrt{\big\langle (s-t), G_\varphi(s)(s-t)\big\rangle}}\ddd t\\
	&\quad +\big(\psi_2 T_{k_{\l,2}}(\psi_1f)\big)\big(\varphi(s)\big), 
\end{aligned}
\end{equation}
with some kernel $k_{\l,2}$ bounded near $0$, and we again deduce $T_{k_{\l,2}}\in\Op\sS^{-2}_1$.

Consider the homogeneous function
\[
h_a: \rr^2\ni t\mapsto \dfrac{1}{4\pi \sqrt{\big\langle t, G_\varphi(a) t\big\rangle}},
\]
then the equality \eqref{tmp1} takes the form
\begin{multline*}
\Big[\big(\psi_2 K_m(\psi_1f) \big)\Big](\varphi(s))
=\psi_2\big(\varphi(s)\big) g_\varphi(s) \Big[h_s\ast \big((\psi_1 f)\circ\varphi\big)\Big](s)
+\big(\psi_2 T_{k_{\l,2}}(\psi_1g)\big)\big(\varphi(s)\big),
\end{multline*}
where $\ast$ is the convolution product on $\rr^2$. This shows
that the principal symbol in the chosen coordinates is $g_\varphi(s) \,\widehat{h_s}(\xi)$.
The point $x$ corresponds to $s=0$, and $G_\varphi(0)=I_2$ and $g_\varphi(0)=1$. Then
\[
h_0(t)=\dfrac{1}{4\pi |t|},\quad \widehat{h_0}(\xi)=\dfrac{1}{2|\xi|},
\]
so we obtain that the principal symbol of $K_m$ at $x$ is $p_{K_{m}}(x,\xi)=\frac{1}{2|\xi|}$.

As the principal symbol of $K_\lambda$ does not vanish for $\xi\ne 0$, the operator $K_\l$ is elliptic of order $(-1)$,
so $K_\l:H^s(\S)\to H^{s+1}(\S)$ is Fredholm of index $0$ for any $s\in\rr$. As we have already shown that $K_\l$ is injective on $L^2(\S)$, we then deduce that $K_\l:H^s(\S)\to H^{s+1}(\S)$ is an isomorphism and the proof is complete.
\end{proof}	

Now let get back to the condition $0\in\spece\widetilde \Lambda_\l$ with $\Tilde\Lambda_\l$ from \eqref{lll}. In view of Lemma~\ref{lem5} the well-defined operator
\[
L_\lambda:=K_\lambda^{-\frac{1}{2}}, \quad \l\in\big[-|m|,|m|\big]
\]
is an isomorphism between $H^s(\S)$ and $H^{s-\frac{1}{2}}(\S)$ for any $s\in\rr$. Then  $L_\lambda \widetilde L^{-1}$ is an isomorphism
of $H^s(\S)$ for any $s\in\S$, and the condition $0\in\spece\widetilde \Lambda_\l$
becomes equivalent to
\begin{equation}
	\label{spll}
	0\in\spece \Lambda_\l
\end{equation}
with the operators $\Lambda_\l$	in $L^2(\S;\cc^4)$ given by
\begin{equation*}
	\label{lll2}
\begin{aligned}	
\Lambda_\l&:=\big[(L_\l \widetilde L^{-1})\otimes I_4\big]	\widetilde\Lambda_\l \big[(\widetilde L^{-1} L_\lambda)\otimes I_4\big]\\
&=\begin{pmatrix}
	\Big[L_\lambda\big( \tfrac{\ep-\mu}{4}+(\l+m)K_\l\big)L_\lambda\Big]\otimes I_2 & (L_\l\otimes I_2)W_\l (L_\lambda\otimes I_2)\\
	(L_\lambda\otimes I_2)W_\l (L_\l\otimes I_2) & \Big[L_\l\big( \tfrac{\ep+\mu}{4}+(\l-m)K_\l\big)L_\lambda\Big]\otimes I_2
\end{pmatrix}\\
&=:\begin{pmatrix}
	\Lambda_\l^{11} & \Lambda_\l^{12}\\
	\Lambda_\l^{21} & \Lambda_\l^{22}
\end{pmatrix}
\end{aligned}
\end{equation*}
defined first on $C^\infty(\S,\cc^2)$ and then extended by taking the closure, which is equivalent to taking the maximal domain as it is a first-order pseudodifferential operator.

We know that $K_\l\ge 0$ by Lemma~\ref{lem5}. For $\lambda\in\big(-|m|,|m|\big)$
the numbers $\lambda+m$ and $\lambda-m$ have opposite signs, while the numbers
$\ep+\mu$ and $\ep-\mu$ have the same sign due to
$(\ep+\mu)(\ep-\mu)\equiv\ep^2-\mu^2=4$. To be definite we assume that $\ep+\mu$ and $\lambda+m$ have the same signs, then
$\Lambda_\l^{11}$ is invertible for all $\lambda\in\big(-|m|,|m|\big)$ (if $\ep+\mu$ and $\lambda+m$ have opposite signs,
then one procceds in the same way using the invertibility of $\Lambda_\l^{22}$). It follows
by the theory of block operator matrices \cite[Thm.~2.4.6]{Tretter} that \eqref{spll} holds if and only if
$0\in \spece S_\l$, where $S_\l$ is the Schur complement of $\Lambda_\l$, which is defined by
\begin{align*}
	S_{\l}&=  \Lambda_\l^{22}-  \Lambda_\l^{21}\big(\Lambda_\l^{11}\big)^{-1}\Lambda_\l^{12}\\
	&\equiv  \tfrac{1}{\ep-\mu}L_\l^2 + (\l-m)L_\l K_{\l}L_\l -L_\l W_{\l}L_\l\left(\tfrac{1}{\ep+\mu}L_\l^2 +(\l+m)L_\l K_{\l}L_\l\right)^{-1}L_\l W_{\l}L_\l
\end{align*}
and viewed as an operator in $L^2(\S,\cc^2)$ defined first on $C^\infty(\S,\cc^2)$ and then extended by taking the closure;
for simplicity of writing from now on we denote $L_\lambda\otimes I_2$ and $K_\l\otimes I_2$ again by $L_\l$ and $K_\l$ respectively.
In view of \eqref{hl1}, the preceding discussion shows that for any $\l\in\big(-|m|,|m|\big)$ one has the equivalence
\begin{equation}
\label{lll3}
	\lambda\in \spece H\quad \Longleftrightarrow\quad 0\in\spece S_\l.
\end{equation}

\section{Principal symbols}\label{ss-imp}

In order to study $\spece S_\lambda$ we need to better understand the operators $W_\l$ defined in \eqref{Opkernels} and the anticommutator
\begin{align*}%\label{KerStein}
Z_{\l}&:=(\sigma\cdot \nu) W_{\l}+W_{\l}(\sigma\cdot \nu).
\end{align*}
From the previous consideration one has at least $Z_\l\in \Op\sS^{0}_2$, but we are going to see that actually
$Z_\l\in \Op\sS^{-1}_2$. We collect all observations in the following theorem.
\begin{theorem}\label{main3} Let $\l\in[-|m|,|m|]$, then the principal symbols
	 $p_{W_\l}$ of $W_\l$ and $p_{Z_\l}$ of $Z_\l$ in the special coordinates	defined in  Remark~\ref{rmk4} near $x\in\S$ are
	\begin{align*}
		p_{W_\l}(x,\xi)&=\frac{1}{2}\,\sigma\cdot\left(\frac{\xi_1 e_1 +\xi_2 e_2}{|\xi|}\right),\\
	p_{Z_\l}(x,\xi)&= \frac{1}{2}\big(\k_1(x)-\k_2(x)\big)(\sigma\cdot \nu(x))\,	\frac{\xi_1 \xi_2}{|\xi|^3}.
\end{align*}
\end{theorem}

\begin{proof}
The computations are very close to the respective part of the proof in Lemma~\ref{lem5}.
For a convolution kernel $k$ let $T_{k}$ denote the associated integral operator.
Recall that for small $|x|$ we have
\begin{align*}
	w_\l(x)&\equiv e^{-\sqrt{m^2-\l^2}|x|}\left( 1+\sqrt{m^2-\l^2}|x|\right)i\sigma\cdot\frac{x}{4\pi |x|^3}\\
	&=\big(1-\sqrt{m^2-\l^2}|x|+O(|x|^2)\big)\left( 1+\sqrt{m^2-\l^2}|x|\right)i\sigma\cdot\frac{x}{4\pi |x|^3}\\
	&=\big(1+O(|x|^2)\big)i\sigma\cdot\frac{x}{4\pi |x|^3}=w_m(x)+w_{\l,1}(x)
\end{align*}	
with $|w_{\l,1}(x)|=O(1)$ for  $|x|\to 0$, therefore, $T_{w_{\l,1}}\in \Op\sS^{-2}_2$.
As $\sigma\cdot \nu\in\Op \sS^0_2$, it is sufficient to show the assertions for $W_m$ and $Z_m$ only.

Let $x\in \S$. We again introduce a special local chart  $(U,V,\var)$ as in Remark \ref{rmk4} and make use of \eqref{gg1} and \eqref{gg2}.
Take any $\psi_j\in C^\infty_c(V)$, $j\in\{1,2\}$, and recall the explicit form of the kernel $w_m$,
\[
w_m(y)=\dfrac{1}{4\pi |y|^3}\,i\sigma\cdot y.
\]
For any $s\in U$ we have
\begin{multline*}
\big[\psi_2 W_m(\psi_1f)\big]\big(\varphi(s)\big)\\
=\psi_2\big(\varphi(s)\big))\pv \int_{U}i\sigma\cdot\frac{\varphi(s)-\varphi(t)}{4\pi\big|\varphi(s)-\varphi(t)\big|^3}\psi_1(\varphi(t))f(\varphi(t)) g_\varphi(t)\ddd t\\
+\big[\psi_2 T_{w_{\l,1}}(\psi_1f)\big]\big(\varphi(s)\big),
 \end{multline*}
and using \eqref{gg1} and \eqref{gg2} we rewrite it as 
\begin{multline*}
	\big[\psi_2 W_m(\psi_1f)\big]\big(\varphi(s)\big)\\
	=\psi_2\big(\varphi(s)\big)g_\varphi(s)\pv \int_{U}i\sigma\cdot\frac{D\varphi(s)(s-t)}{4\pi\big\langle (s-t), G_\varphi(s)(s-t)\big\rangle^{\frac{3}{2}}}\psi_1(\varphi(t))f(\varphi(t)) \ddd t\\
	+\big[\psi_2 T_{w_{\l,2}}(\psi_1f)\big]\big(\varphi(s)\big),
\end{multline*}
for some kernel $w_{\l,2}$ with $|w_{\l,2}(x)|=O(|x|^{-1})$ for  $|x|\to 0$, so $T_{w_{\l,2}}\in\Op\sS^{-1}_2$.
If one considers the following matrix-valued homogeneous distribution $h_a $ on $\rr^2$,
\[
h_a:=\pv i\sigma\cdot \frac{D\varphi(a)t}{4\pi\big\langle t, G_\varphi(a) t\big\rangle^{\frac{3}{2}}},
\]
the above can be rewritten as
\begin{multline*}
	\Big[\big(\psi_2 W_m(\psi_1f) \big)\Big](\varphi(s))=\psi_2\big(\varphi(s)\big) g_\varphi(s) \Big[h_s\ast \big((\psi_1 f)\circ\varphi\big)\Big](s)
	+\big(\psi_2 T_{w_{\l,2}}(\psi_1g)\big)\big(\varphi(s)\big),
\end{multline*}
and the principal symbol of $W_\l$ in the chosen coordinates is $g_\varphi(s)\widehat{h_a}(\xi)$.
Now consider the symbol at $x$, i.e. for $s=0$. Due to the special choice of coordinates
one has $G_\var(0)=I_2$ and $g_\varphi(0)=1$. In addition, $D\varphi(0)t=t_1e_1+t_2e_2$,
which gives
\[
h_0=\pv i\sigma\cdot \frac{t_1e_1+t_2e_2}{4\pi |t|^3}.
\]
Using 
\begin{equation}
	   \label{fff}
\Big(\pv\dfrac{t_j}{2\pi|t|^3}\Big)^{\widehat{\phantom{aa}}}(\xi)=-i\dfrac{\xi_j}{|\xi|}, \quad j\in\{1,2\},
\end{equation}
we obtain the required form of the principal symbol of  $w_m$ at $x$.

To analyze $Z_m$ we remark first that it is given by \[
Z_m f(x)=\lim_{\rho\searrow 0}\int_{\substack{y\in\S\\ |x-y|>\rho}} z_m(x,y) f(y)\ddd s(y)
\]
with the singular kernel
\begin{align*}
z_m(x,y)&=\big(\sigma\cdot\nu(x)\big) w_m(x-y)+ w_m(x-y) \big(\sigma\cdot\nu(y)\big)\\
&=\dfrac{i}{4\pi |x-y|^3}\,\Big( \big(\sigma\cdot\nu(x)\big) \, \sigma\cdot(x-y) + \sigma\cdot (x-y)\big(\sigma\cdot\nu(y)\big)\Big).
\end{align*}
Using \eqref{ssxy} we transform
\begin{align*}
	 \big(\sigma\cdot\nu(x)\big) &\, \sigma\cdot(x-y) + \sigma\cdot (x-y)\big(\sigma\cdot\nu(y)\big)\\
	 &=\big\langle \nu(x),x-y\big\rangle I_2 +i\sigma\big( \nu(x)\times (x-y)\big)\\
	 &\quad + \big\langle x-y,\nu(y)\big\rangle I_2 +i\sigma\cdot \big ( (x-y)\times\nu(y)\big)\\
	 &=\big\langle \nu(x)+\nu(y), x-y\big\rangle I_2 +i\sigma\cdot \big[\big(\nu(x)-\nu(y)\big)\times (x-y)\big].
\end{align*}	
This gives the representations
\begin{align*}
	z_m(x,y)&=\theta(x,y)+\theta_1(x,y),\\
	\theta(x,y)&:=-\frac{1}{4\pi|x-y|^3}\,\sigma\cdot \big[\big(\nu(x)-\nu(y)\big)\times (x-y)\big],\\
	\theta_1(x,y)&:=\dfrac{\big\langle \nu(x)+\nu(y), x-y\big\rangle}{|x-y|^3}\,I_2,
\end{align*}	
and $Z_m=T_\theta+T_{\theta_1}$. Now remark that $T_{\theta_1}=(N'-N)\otimes I_2$,
where
\[
N f(x)=\int_\S \dfrac{\big\langle\nu(y),y-x\big\rangle}{4\pi|x-y|^3}f(y)\ddd s(y),
\quad
N' f(x)=-\int_\S \dfrac{\big\langle\nu(x),y-x\big\rangle}{4\pi|x-y|^3} f(y)\ddd s(y),
\]
are the so-called Neumann-Poincar\'e operator and its formal adjoint. For a recent detailed study of $N$ and a review of available results we refer to the recent works \cite{Miy,MR,MR2}. In particular, $N\in \Op \sS^{-1}_1$ with an explicitly known real-valued symbol \cite{Miy}, so $N'$ has the same principal symbol as $N$, and it follows that $T_{\theta_1}\in \Op\sS^{-2}_2$.

In order to compute the principal symbol of $T_\theta$ we again use the above special coordinates $(U,V,\varphi)$ near some $x\in \Sigma$
and all associated objects. Let $\psi_j\in C^\infty_c(V)$, $j\in\{1,2\}$, then for any $f\in C^\infty(\S,\cc^2)$ we have
\[
\big[\psi_2 T_\theta (\psi_1f)\big]\big(\varphi(s)\big)
=\psi_2\big(\varphi(s)\big)) \int_{U}\theta\big(\varphi(s),\varphi(t)\big)\psi_1(\varphi(t))f(\varphi(t)) g_\varphi(t)\ddd t,
\]
Using \eqref{gg2} and
$
\nu\big(\varphi(s)\big)-\nu\big(\varphi(t)\big)=D(\nu\circ \varphi)(s)(s-t)+O\big(|s-t|^2\big)
$
we obtain the representation
\[
\theta\big(\varphi(s),\varphi(t)\big)\\
=-\sigma\cdot\frac{\big(D(\nu\circ \varphi)(s)(s-t) \big)\times \big(D\varphi(s)(s-t)\big)}{4\pi\big\langle (s-t), G_\varphi(s)(s-t)\big\rangle^{\frac{3}{2}}}+r(s,t)
\]
with some $r$ bounded near the diagonal $s=t$. If we introduce the homogeneous function
\[
\rho_a:\rr^2 \ni t \mapsto -\sigma\cdot\frac{\big(D(\nu\circ \varphi)(a) t \big)\times \big(D\varphi(a)t\big)}{4\pi\big\langle t, G_\varphi(a) t\big\rangle^{\frac{3}{2}}},
\]
then we arrive at
\begin{multline*}
	\Big[\big(\psi_2 T_\theta(\psi_1f) \big)\Big](\varphi(s))=\psi_2\big(\varphi(s)\big) g_\varphi(s) \Big[\rho_s\ast \big((\psi_1 f)\circ\varphi\big)\Big](s)
	+\big(\psi_2 T_{\Tilde \theta}(\psi_1g)\big)\big(\varphi(s)\big),
\end{multline*}
with some bounded kernel $\Tilde\theta$, so $T_{\Tilde \theta}\in \Op\sS^{-2}_2$. This shows that the principal kernel
of $T_\theta$ in the chosen coordinates is $g_\varphi(s)\widehat{\rho_s}(\xi)$.

To find a more explicit expression at $x$ we set $s=0$. Recall that by the choice of $\varphi$ we have
$G_\var(0)=I_2$ and $g_\varphi(0)=1$, $D\varphi(0)t=t_1e_1+t_2e_2$, and, in addition,
$D(\nu\circ \varphi)(0)t= \kappa_1(x)\, t_1 e_1+\kappa_2(x)\, t_2 e_2$. Then
\begin{align*}
\rho_0(t)&=-\sigma\cdot\frac{\big(\kappa_1(x)\, t_1 e_1+\kappa_2(x)\, t_2 e_2\big)\times \big(t_1e_1+t_2e_2\big)}{4\pi |t|^3}\\
&\equiv -\big(\kappa_1(x)-\kappa_2(x)\big) \dfrac{t_1 t_2}{4\pi |t|^3}\, (\sigma\cdot \nu(x)),
\end{align*}
where we used $e_1\times e_2=\nu(x)$. By applying the Fourier tranform on the both sides of \eqref{fff}
we arrive at
\begin{align*}
\widehat{\dfrac{t_j}{|t|}}&=-2\pi i \pv \dfrac{\xi_j}{|\xi|^3},\quad j\in\{1,2\},\\
\Big(\dfrac{t_1 t_2}{4\pi |t|^3}\Big)^{\widehat{\phantom{aa}}}(\xi)
&=\dfrac{1}{4\pi}\Big(-\partial_1 \dfrac{t_2}{2\pi |t|}\Big)^{\widehat{\phantom{aa}}}(\xi)
=\dfrac{1}{4\pi} (-i \xi_1) \Big( \dfrac{t_2}{ |t|}\Big)^{\widehat{\phantom{aa}}}(\xi)\\
&=\dfrac{1}{4\pi} (-i \xi_1) (-2\pi i) \pv \dfrac{\xi_2}{|\xi|^3}=-\dfrac{\xi_1\xi_2}{2|\xi|^3}.
\end{align*}
Therefore, the principal symbol of $T_\theta$  (and of all $Z_\l$) in the chosen coordinates is
\[
p_{Z_\l}(x,\xi)=\widehat{\rho_0}(\xi)\equiv \frac{1}{2}\big(\k_1(x)-\k_2(x)\big)(\sigma\cdot\nu(x))\,
\frac{\xi_1 \xi_2}{|\xi|^3}.\qedhere
\]
\end{proof}

\begin{remark}
The operators $W_\l$ are closely related to Poincar\'e-Steklov operators for Dirac operators as discussed recently in \cite{BBZ}.
The operator $Z_m$ is sometimes referred to as Kerzman-Stein operator, as it represents
the imaginary part of the Clifford algebra-valued Hilbert transform on $\S$, see e.g.~\cite[Sec.~7]{BS}.
\end{remark}

\section{Analysis of the Schur complement}\label{ss-schur}

We are going to apply all preceding symbolic computation in order to study the spectrum of the Schur complement $S_\l$.
\begin{lemma}\label{equi2}
One has $S_\l\in \Op\sS^0_2$. Furthermore, consider the operators
\begin{equation*}
	  \label{qr}
Q_{\l}:=(\sigma\cdot \nu)Z_\l W_\l \in \Op\sS^{-1}_2, \qquad R_\l:=L_\l Q_\l L_\l\in \Op\sS^{0}_2,
\end{equation*}
with $R_{\l}$ viewed as a bounded operator in $L^2(\S,\cc^2)$, then for any $\l\in\big(-|m|,|m|\big)$
one has the equivalence
\[
0\in \spece S_\l \quad \Longleftrightarrow\quad 
\l\in \spece\Big( \dfrac{2}{\ep}\,R_\l -\frac{\mu m}{\ep}\Big).
\]
\end{lemma}

\begin{proof}
Recall that we have $(\ep+\mu)(\lambda+m)>0$ and $L_\l=K_\l^{-\frac{1}{2}}$, therefore,
\begin{align*}
	S_{\l}&=\tfrac{1}{\ep-\mu}L_\l^2 + (\l-m)L_\l K_{\l}L_\l -L_\l W_{\l}L_\l\left(\tfrac{1}{\ep+\mu}L_\l^2 +(\l+m)L_\l K_{\l}L_\l\right)^{-1}L_\l W_{\l}L_\l\\
	&\equiv \tfrac{1}{\ep-\mu}L_\l^2 + (\l-m)I_2 -L_\l W_{\l}L_\l \left(\tfrac{1}{\ep+\mu}L_\l^2 +(\l+m)I\right)^{-1}L_\l W_{\l}L_\l\\
	&\equiv \tfrac{1}{\ep-\mu}L_\l^2 + (\l-m)I_2 -(\ep+\mu)L_\l W_{\l} \Big[1+(\ep+\mu)(\l+m)L_\l^{-2}\Big]^{-1} W_{\l}L_\l.
\end{align*}
For $a:=(\ep+\mu)(\l+m)>0$ we can represent
\begin{align*}
\Big[1+aL_\l^{-2}\Big]^{-1}&= (1+a L_\l^{-2}-aL_\l^{-2}) \Big[1+a L_\l^{-2}\Big]^{-1}\\
&=I-aL_\l^{-2}\Big[1+a L_\l^{-2}\Big]^{-1}\\
\text{(iterate) }&=I-aL_\l^{-2}\bigg(I-aL_\l^{-2}\Big[1+a L_\l^{-2}\Big]^{-1} \bigg)\\
&=I-aL_\l^{-2}+a^2L_\l^{-4}\Big[1+aL_\l^{-2}\Big]^{-1}.
\end{align*}	
Recall that $W_\l\in \Op\sS^{0}_2$ and $L_\l\in \Op\sS^{\frac{1}{2}}_2$, so $L_\l^{-4}\Big[1+aL_\l^{-2}\Big]^{-1}\in \Op\sS^{-2}_2$,
and the substitution into the above expression of $S_\l$ gives, with some $B_0\in\Op\sS^{-1}_2$,
\begin{align*}
S_\l&=\tfrac{1}{\ep-\mu}L_\l^2 + (\l-m)I_2 -(\ep+\mu)L_\l W_{\l} \Big(I-(\ep+\mu)(\l+m) L_\l^{-2} \Big)W_{\l}L_\l+B_0\\
&\equiv \tfrac{1}{\ep-\mu}L_\l^2 + (\l-m)I_2-(\ep+\mu)L_\l W_\l^2 L_\l +(\ep+\mu)^2(\l+m)L_\l W_\l L_\l^{-2}W_\l L_\l +B_0.
\end{align*}
As $L_\l$ is a scalar operator, so using the commutators one obtains $L_\l W_\l L_\l^{-2}W_\l L_\l=W_\l^2+B_1$
for some $B_1\in \Op\sS^{-1}_2$. Note that the principal symbol $p_{W_\l}$ of $W_\l$ satisfies $p_{W_\l}=\frac{1}{4}\,I_2$ (see Theorem \ref{main3}), so $W_\l^2=\frac{1}{4}I_2+B_2$ for some $B_2\in \Op \sS^{-1}_2$
and then
\[
L_\l W_\l L_\l^{-2}W_\l L_\l=\frac{1}{4}\,I+B_1+B_2.
\]
Taking into account $\frac{1}{\ep-\mu}=\frac{\ep+\mu}{4}$, the last expression for $S_\l$ can be rewritten as
\[
S_\l=(\ep+\mu)\,\frac{1}{4} L_\l^2+(\lambda-m)-(\ep+\mu)L_\l W_\l^2 L_\l +\frac{1}{4} (\ep+\mu)^2(\lambda+m)+B_3
\]
with $B_3:=B_0+ (\ep+\mu)^2(\lambda+m)(B_1+B_2)\in \Op\sS^{-1}_2$. This already shows that $S_\l$ has order zero, so it is defined on $L^2(\S,\cc^2)$.
Recall that adding a compact operator does not change the essential spectrum and that all operators in $\Op \sS^{-1}_2$ are compact. We conclude that the condition $0\in\spece S_\l$ is equivalent to
\begin{equation}
	  \label{tmp3}
0\in \spece \bigg[(\ep+\mu)L_\l\Big( \frac{1}{4}-W_\l^2\Big)L_\l + \Big(\lambda-m+\frac{(\ep+\mu)^2(\l+m)}{4}\Big)\bigg].
\end{equation}
We use $\ep+\mu=\frac{4}{\ep-\mu}$ to compute
\[
\l-m+\frac{(\ep+\mu)^2(\l+m)}{4}=\l-m+\frac{\ep+\mu}{\ep-\mu}(\l+m)=\dfrac{2}{\ep-\mu}(\lambda\ep+m\mu ),
\]
and the condition \eqref{tmp3} takes the form	
\[
	0\in \spece \bigg[(\ep+\mu)L_\l\Big( \frac{1}{4}-W_\l^2\Big)L_\l  +\dfrac{2}{\ep-\mu}(\lambda\ep+m\mu )\bigg].
\]
If one multiplies the operator on the right-hand side by $\frac{\ep-\mu}{2}$ and takes into account $\ep^2-\mu^2=4$, one arrives
at
\[
	0\in \spece \Big[2 L_\l\Big( \frac{1}{4}-W_\l^2\Big)L_\l + (\lambda\ep+m\mu)\Big]
\]
which is equivalent to
\begin{equation}
	\label{tmp4}
	\l\in\spece \Big[\,\frac{2}{\ep}L_\l\Big( W_\l^2-\frac{1}{4}\Big)L_\l-\frac{m \mu }{\ep}\,	\Big].
\end{equation}
Using the identity \eqref{Pro W et K} we obtain
\begin{align*} 
W_\l^2-\frac{1}{4}&=W_\l^2+((\sigma\cdot \nu)W_{\l}\big)^2 + (m^2-\l^2) \big((\sigma\cdot \nu)K_{\l}\big)^2\\
&\equiv W_\l^2+((\sigma\cdot \nu)W_{\l}\big)^2\mod \Op \sS^{-2}_2\\
&\equiv (\sigma\cdot\nu)(\sigma\cdot\nu)W_\l W_\l+(\sigma\cdot\nu)W_\l(\sigma\cdot\nu)W_\l \mod \Op \sS^{-2}_2\\
&\equiv(\sigma\cdot\nu)\Big[ (\sigma\cdot\nu) W_\l + W_\l (\sigma\cdot\nu)\Big]W_\l \mod \Op \sS^{-2}_2\\
&\equiv Q_\lambda \mod \Op \sS^{-2}_2,
\end{align*}
which, by definition, gives $L_\l\Big( W_\l^2-\frac{1}{4}\Big)L_\l=R_\l \mod \Op \sS^{-1}_2$, and the substitution into \eqref{tmp4} concludes the proof.
\end{proof}

Let us recall where we are standing: the characterization \eqref{lll3} of $\spece H$ together with Lemma~\ref{equi2} show that
for any $\lambda\in \big(-|m|,|m|\big)$ one has the equivalence
\begin{equation}
	   \label{hhrr}
\lambda\in\spece H \quad \Longleftrightarrow\quad \l\in \spece\Big( \dfrac{2}{\ep}\,R_\l -\frac{\mu m}{\ep}\Big),
\end{equation}
while $R_\l\in \Op\sS^0_2$.

Let us recall how to compute the essential spectrum of a zero-order pseudodifferential operator: the result is folkloric and is mentioned e.g. in \cite{MR2} (in the discussion just before Theorem 2.1), \cite[Thm. 2.1]{cdv} or \cite[Prop. 1.1.5]{adams}, but we prefer to state it explicitly.

\begin{lemma}\label{lem8}
Let $M$ be a compact Riemannian manifold, $V$ a smooth finite-dimensional vector bundle over $M$, and
$B:\mathit{L}^2(M,V)\to \mathit{L}^2(M,V)$ a classical pseudodifferential operator of order zero
with principal symbol $b_0$. Then
\[
\spece B=\bigcup_{x\in M}\bigcup_{\xi\in T^*_x M\setminus\{0\}} \spec b_0(x,\xi).
\]
\end{lemma}

\begin{proof} The argument is borrowed directly from \cite[Sec.~2.3]{MR2}.
For a classical pseudodifferential operator $T:\mathit{H}^n(M,V)\to \mathit{L}^2(M,V)$ of order $n$
the Fredholmness is equivalent to the ellipticity,
i.e. to the invertibility of the principal symbol. In our case $n=0$ and the principal symbol of $B-z$ is $b_0-z$, therefore,
\begin{align*}
\spece B=\big\{ z\in \cc:\,&\  B-z \text{ is not Fredholm}\big\}\\
=\big\{ z\in \cc:\,&\  B-z \text{ is not elliptic}\big\}\\
=\big\{ z\in \cc:\, & \ b_0(x,\xi)-z \text{ is not invertible}\\
& \text{ for some $x\in M$ and $\xi\in T^*_x M\setminus\{0\}$}\big\},\\
= \big\{ z\in \cc:\, &\ z\in \spec b_0(x,\xi)\text{ for some $x\in M$ and $\xi\in T^*_x M\setminus\{0\}$}\big\}. \qedhere
\end{align*}
\end{proof}
 
In view of \eqref{hhrr} it is sufficient to compute the essential spectrum of
\[
R_\l\equiv L_\l(\sigma\cdot \nu) Z_\l W_\l L_\l,
\]
with the help of Lemma~\ref{lem8}. Its principal symbol $p_{R_\lambda}$
is the product of the principal symbols of the five factors. If $x\in\Sigma$ and one chooses the special local coordinates as in Remark \ref{rmk4}, one obtains with the help of Theorem \ref{main3}:
\begin{align*}
p_{R_\l}(x,\xi)&= \sqrt{2|\xi|} \big(\sigma\cdot \nu(x)\big) \frac{1}{2}\big(\k_1(x)-\k_2(x)\big)\big(\sigma\cdot \nu(x)\big)\,	\frac{\xi_1 \xi_2}{|\xi|^3}\\
&\quad\quad \cdot 
\frac{1}{2}\,\sigma\cdot\left(\frac{\xi_1 e_1 +\xi_2 e_2}{|\xi|}\right)
\sqrt{2|\xi|}\\
&=\dfrac{\k_1(x)-\k_2(x)}{2}\,\dfrac{\xi_1 \xi_2}{|\xi|^3}\, \sigma\cdot (\xi_1 e_1 +\xi_2 e_2).
\end{align*}
For any $a\in\rr^3$ the eigenvalues of $\sigma\cdot a$ are $\pm|a|$. As
$|\xi_1 e_1 +\xi_2 e_2|=|\xi|$, for any $x\in\Sigma$ and $\xi\ne 0$ one has
\[
\spec p_{R_\lambda}(x,\xi)=\bigg\{\dfrac{\k_1(x)-\k_2(x)}{2}\,\dfrac{\xi_1 \xi_2}{|\xi|^2},-\dfrac{\k_1(x)-\k_2(x)}{2}\,\dfrac{\xi_1 \xi_2}{|\xi|^2}\bigg\},
\]
therefore,
\begin{align*}
\bigcup_{\xi\ne 0} \spec p_{R_\lambda}(x,\xi)&=\bigg[-\dfrac{\big| \k_1(x)-\k_2(x)\big|}{4},\dfrac{\big| \k_1(x)-\k_2(x)\big|}{4}\bigg].
\end{align*}
Taking now the union over all $x\in\Sigma$ we arrive at
\begin{gather*}
	\spece R_\lambda=\Big[-\dfrac{A_\Sigma}{4},\dfrac{A_\Sigma}{4}\bigg],\quad 	A_\Sigma:=\max_{x\in\Sigma} \big| \k_1(x)-\k_2(x)\big|,
\end{gather*}
and the substitution into \eqref{hhrr} gives
\[
\big(-|m|,|m|\big)\cap \spece H=\big(-|m|,|m|\big)\cap
\Big[-\dfrac{\mu m}{\ep}-\dfrac{A_\Sigma}{2|\ep|},-\dfrac{\mu m}{\ep}+\dfrac{A_\Sigma}{2|\ep|}\bigg],
\]
which finishes the proof of Theorem~\ref{main1}.

\section*{Acknowledgments} The work is funded by the Deutsche Forschungsgemeinschaft (DFG, German Research Foundation) -- 491606144. The authors are thankful to Daniel Grieser for valuable discussions of aspects related to pseudodifferential operators.
Parts of the text were written during the stay of K.P. at the Erwin Schr\"odinger Institute in Vienna in November 2022, 
and we thank the institute for the support provided.

%===============================================================================================================================

%\subsection*{ Data availability statement} Data sharing is not applicable to this article as no new data were created or analyzed in this study.
%\subsection*{Conflict of interest} The authors have no conflicts of interest to disclose. 
	%=================================================================================================
%=================================================================================================

\end{document}